\documentclass[12pt,oneside,english]{amsart}
\usepackage[T1]{fontenc}
\usepackage[utf8]{inputenc}
\usepackage{geometry}
\usepackage{color}
\usepackage{babel}
\usepackage{amstext}
\usepackage{amsthm}
\usepackage{amssymb}
\usepackage{esint}
\usepackage[unicode=true,pdfusetitle,
 bookmarks=true,bookmarksnumbered=false,bookmarksopen=false,
 breaklinks=false,pdfborder={0 0 0},pdfborderstyle={},backref=false,colorlinks=true]
 {hyperref}
\usepackage[numbers]{natbib}

\makeatletter
\numberwithin{equation}{section}
\numberwithin{figure}{section}
\theoremstyle{plain}
\newtheorem{thm}{\protect\theoremname}

\allowdisplaybreaks

\makeatother

\providecommand{\theoremname}{Theorem}

\begin{document}
\global\long\def\essinf{\operatornamewithlimits{ess\,inf}}%

\global\long\def\esssup{\operatornamewithlimits{ess\,sup}}%

\global\long\def\supp{\operatornamewithlimits{supp}}%

\title{Upper endpoint estimates and extrapolation for commutators}
\author{Kangwei Li}
\address{(Kangwei Li) Center for Applied Mathematics, Tianjin University, Tianjin 300072, The People’s Republic of China} 
\email{kli@tju.edu.cn}
\author{Sheldy Ombrosi}
\address{(Sheldy Ombrosi) Departamento de An\'alisis Matem\'atico y Matem\'atica Aplicada. Facultad de Ciencias Matem\'aticas. Universidad Complutense
(Madrid, Spain). Departamento de Matem\'atica e Instituto de Matem\'atica de Bah\'ia Blanca. Universidad Nacional
del Sur - CONICET (Bah\'{\i}a Blanca, Argentina).}
\email{sombrosi@ucm.es}
\author{Israel P. Rivera-Ríos}
\address{(Israel P. Rivera R\'ios) Departamento de An\'alisis Matem\'atico, Estad\'{\i}stica e Investigaci\'on Operativa
y Matem\'atica Aplicada. Facultad de Ciencias. Universidad de M\'alaga
(M\'alaga, Spain). Departamento de Matem\'atica. Universidad Nacional
del Sur (Bah\'{\i}a Blanca, Argentina).}
\email{israelpriverarios@uma.es}

\thanks{
The first author was supported by the National Natural Science Foundation of China (Grant No. 12001400). The second and the third authors were partially supported by FONCyT PICT 2018-02501. 
The second author was partially supported by Spanish Government Ministry of Science and Innovation through grant PID2020-113048GB-I00.
The third author was partially supported as well by FONCyT PICT 2019-00018 and by Junta de Andaluc\'ia UMA18FEDERJA002.}
\maketitle
\begin{abstract}In this note we revisit the upper endpoint estimates for commutators following the line
by Harboure, Segovia and Torrea \cite{HST}. Relying upon the suitable $BMO$ subspace suited for the commutator
that was introduced by Accomazzo \cite{A}, we obtain a counterpart for commutators of the upper endpoint extrapolation result by Harboure, Mac\'ias and Segovia \cite{HMS}. Multilinear counterparts
are provided as well.
\end{abstract}

\section{Introduction and main results}

Extrapolation has been a fruitful area of research since the 80s. First
results in that direction were due to Rubio de Francia \cite{RdF1,RdF2}.
We briefly discuss the general principle behind that kind of results
in the following lines.

We say that $w$ is a weight if it is a non-negative locally integrable function on $\mathbb{R}^{n}$.
Recall that $w\in A_{p}$ for $1<p<\infty$ if
\[
[w]_{A_{p}}=\sup_{Q}\fint_{Q}w\left( \fint_Q w^{-\frac{1}{p-1}}\right)^{p-1}<\infty
\]
 and that $w\in A_{1}$ if 
\[
[w]_{A_{1}}=\left\Vert \frac{Mw}{w}\right\Vert _{L^{\infty}}<\infty,
\]
where $M$ stands for the Hardy-Littlewood maximal function 
\[
Mf(x)=\sup_{x\in Q} \fint_{Q}|f(y)|dy,
\]
where each $Q$ is a cube of $\mathbb{R}^{n}$ with its sides parallel
to the axis.

A fundamental property of the $A_{p}$ classes is that they characterize the
weighted $L^{p}$ boundedness of the Hardy-Littlewood maximal operator
and they are good weights for a number of operators in the theory
such as singular integrals, commutators and some further ones.

The Rubio de Francia extrapolation results says that if $T$ is a
sublinear operator such that for some $1<p_{0}<\infty$ 
\begin{equation}
\|Tf\|_{L^{p_{0}}(w)}\leq c_{w,T,p_{0}}\|f\|_{L^{p_{0}}(w)}\label{eq:hypTfp0}
\end{equation}
for every $w\in A_{p_{0}}$, then 
\begin{equation}
\|Tf\|_{L^{p}(w)}\leq c_{w,T,p}\|f\|_{L^{p}(w)}\label{eq:Tfp}
\end{equation}
for every $w\in A_{p}$ and every $1<p<\infty$.

This approach has been extensively studied by a number of authors
in a wide variety of settings. For instance, in the linear setting
there are fundamental works due to Cruz-Uribe, Martell and Pérez \cite{CUMBook,CUMP1,CUMP2,CUMP3,CUM,CUP},
Duoandikoetxea \cite{D1,D2}, Dragicevic, Grafakos, Petermichl, Pereyra
\cite{DGPP}, Harboure, Macías, Segovia \cite{HMS,HMS2}. After a
number of intermediate results in the multilinear setting (see for
instance \cite{GM,CGMS,CUM}) the question was succesfully solved
in the last years, in works such as \cite{LMO,LMMOV,N}.

A useful development in the area since Rubio de Francia's pioneering
works consisted in learning that the operator involved in \eqref{eq:hypTfp0}
and \eqref{eq:Tfp} actually plays no role. To be more precise, it can
be replaced by a condition on pairs of functions. Assume that $\mathcal{F}$
is a family of pairs of functions such that for some $1<p_{0}<\infty$
\begin{equation}
\|f\|_{L^{p_{0}}(w)}\leq c_{w,T,p_{0}}\|g\|_{L^{p_{0}}(w)}\label{eq:hyp(f,g)p0}
\end{equation}
for every $(f,g)\in\mathcal{F}$ and every $w\in A_{p_{0}}$, then
\begin{equation}
\|f\|_{L^{p}(w)}\leq c_{w,T,p}\|g\|_{L^{p}(w)}\label{eq:(f,g)p}
\end{equation}
for every $(f,g)\in\mathcal{F}$, for every $w\in A_{p}$ and every
$1<p<\infty$.

Another line of research would consist in considering the endpoints,
namely $p_{0}=\infty$ or $p_{0}=1$ as a ``departing'' point for
extrapolation. For instance, the following result was obtained in
\cite{GC,CUMBook}
\begin{thm}
\label{thm:ExtrapInfty}Let $(f,g)$ be a pair of functions and suppose
that 
\[
\|gw\|_{L^{\infty}}\le c_{w}\|fw\|_{L^{\infty}}
\]
holds for all $w$ with $w^{-1}\in A_{1}$, where $c_{w}$ depends
only on $[w^{-1}]_{A_{1}}$. Then for all $1<p<\infty$ and all $w\in A_{p}$,
we have 
\[
\|g\|_{L^{p}(w)}\le\tilde{c}_{w}\|f\|_{L^{p}(w)},
\]
where $\tilde{c}_{w}$ depends only on $[w]_{A_{p}}$. 
\end{thm}

There are a number of operators that do not map $L^{\infty}$ into
$L^{\infty}$ such as the Hilbert transform. However for the
Hilbert transform $H$ itself and even for a larger class of operators,
the Calderón-Zygmund operators, it is possible to show that they map
$L^{\infty}$ into $BMO$. Weighted versions of that result were studied
first in \cite{MW}. There it was shown that if $w\in A_{1}$, then
\begin{equation}
 \fint_{Q}|Hf-(Hf)_{Q}|\leq C\left\Vert \frac{f}{w}\right\Vert _{L^{\infty}}\essinf_{Q}w(x).\label{eq:oscH}
\end{equation}
In view of this estimate, it seems natural to think about extending
this result to Calderón-Zygmund operators, and also, within the framework
of extrapolation whether it would be possible to extrapolate from
that weighted $L^{\infty}\rightarrow BMO$ bound in order to obtain
weighted $L^{p}$ estimates. Those questions were answered in the
positive in the inspiring paper \cite{HMS} by Harboure, Mac\'ias and Segovia. In that work the following
extrapolation result was settled.
\begin{thm}
Let $T$ be a sublinear operator defined on $\mathcal{C}_{0}^{\infty}(\mathbb{R}^{n})$
satisfying that for any cube $Q\subset\mathbb{R}^{n}$ and any $w\in A_{1}$
\[
 \fint_{Q}|Tf-(Tf)_{Q}|\leq c_{w,T}\left\Vert \frac{f}{w}\right\Vert _{L^{\infty}}\essinf_{Q}w.
\]
 Then for every $1<p<\infty$ and every $w\in A_{p}$ we have that
\[
\|Tf\|_{L^{p}(w)}\leq c_{w}\|f\|_{L^{p}(w)}.
\]
\end{thm}

Quite recently in \cite{CPR}, a quantitative version of this result
was obtained. In that paper it was shown that if $\delta\in(0,1)$
and
\[
\inf_{c\in\mathbb{R}}\left( \fint_{Q}|Tf-c|^{\delta}\right)^{\frac{1}{\delta}}\leq c_{n,\delta,T}\varphi([w]_{A_{1}})\left\Vert \frac{f}{w}\right\Vert _{L^{\infty}}\essinf_{Q}w,
\]
then 
\[
\|Tf\|_{L^{p}(w)}\leq c_{n}\varphi(\left\Vert M\right\Vert _{L^{p}(w)})\left\Vert M\right\Vert _{L^{p'}(\sigma)}\|f\|_{L^{p}(w)},
\]
where $\sigma=w^{-\frac{1}{p-1}}$. Note that since $\left\Vert M\right\Vert _{L^{p}(w)}\lesssim[w]_{A_{p}}^{\frac{1}{p-1}}$
such an estimate yields that 
\[
\|Tf\|_{L^{p}(w)}\leq c_{n}\varphi([w]_{A_{p}}^{\frac{1}{p-1}})[w]_{A_{p'}}\|f\|_{L^{p}(w)}.
\]
In the same paper it is shown that for Calderón-Zygmund operators
\begin{equation}
\inf_{c\in\mathbb{R}}\left( \fint_{Q}|Tf-c|^{\delta}\right)^{\frac{1}{\delta}}\leq c_{n,\delta,T}[w]_{A_{1}}\left\Vert \frac{f}{w}\right\Vert _{L^{\infty}}\essinf_{Q}w,\label{eq:OscT}
\end{equation}
namely $\varphi(t)=t$ and hence the sharp exponent for the $A_{p}$
constant $\max\left\{ 1,\frac{1}{p-1}\right\} $ for that class is
not recovered. Such a fact is not surprising since current best known extrapolation argument from the lower endpoint neither recovers the sharp estimate.
At this point we would like to note that a way more general version of the aforementioned extrapolation result, replacing $L^p(w)$ spaces by function Banach spaces and the $A_p$ constant by suitable boundedness constants of the maximal function over those spaces, was obtained very recently in \cite{NR}. Also a quantitative multilinear result in that direction was provided in \cite[Corollary 4.14]{N} 

Now we turn our attention to our contribution in this work. We recall
that given $b\in BMO$, the Coifman-Rochberg-Weiss commutator is defined
as
\[
[b,T]f(x)=b(x)Tf(x)-T(bf)(x).
\]
It is well known that $[b,T]$ is bounded on $L^{p}(w)$ and that,
as Pérez showed in \cite{P}, $[b,T]$ is not of weak type $(1,1)$
but it satisfies the following estimate instead
\[
w\left(\left\{ |[b,T]f(x)|>t\right\} \right)\lesssim[w]_{A_{1}}^{2}\log(e+[w]_{A_{1}})\int_{\mathbb{R}^{n}}\Phi\left(\frac{|f|}{t}\right)w
\]
where $\Phi(t)=t\log(e+t)$. The quantitative dependence was obtained
in \cite{LORR}.

In view of (\ref{eq:oscH}) and (\ref{eq:OscT}) one may wonder what
can be said about commutators. In \cite[Theorem A]{HST}, Harboure,
Segovia and Torrea provided the following result.
\begin{thm}
Let $T$ be a Calderón-Zygmund operator and let $b\in BMO$. Then
the following statements are equivalent
\begin{enumerate}
\item For every ball $B$ and every $f\in L_{c}^{\infty}(\mathbb{R}^{n})$,
\begin{equation}
\fint_{B}|[b,T]f(x)-([b,T]f)_{B}|dx\lesssim\left\Vert f\right\Vert _{L^{\infty}}.\label{eq:osc=00005Bb,T=00005DHMS}
\end{equation}
\item The function $b$ satisfies the following condition. For any cube
$Q$ and $u\in Q$
\[
 \Big(\fint_{Q}|b-b_{Q}|\Big)T(f\chi_{(2Q)^{c}})(u)\leq C\|f\|_{L^{\infty}(\mathbb{R}^{n})}
\]
for every $f\in L_{c}^{\infty}(\mathbb{R}^{n})$.
\end{enumerate}
\end{thm}

Also in \cite{HST} the authors point out that if $T$ is the Hilbert
transform and any of the conditions in the preceding Theorem are satisfied,
then necessarily $b$ is constant and hence $[b,H]=0$. This fact leads us to think about the possibility
of considering a ``smaller'' oscillation in the left hand side of
(\ref{eq:osc=00005Bb,T=00005DHMS}).

Aiming for a dual of the Hardy spaces for commutators studied by Pérez
\cite{P} and Ky \cite{Ky}, Accomazzo introduced in \cite{A} the
spaces $BMO_{b}^{q}$ which are defined as follows. Given a function
$b$, and $q\in[1,\infty)$ we have that $f\in BMO_{b}^{q}$ if 
\[
\|f\|_{BMO_{b}^{q}}:=\sup_{B}\left(\inf_{c_{0},c_{1}\in\mathbb{R}}\fint_{B}|f(x)-c_{0}+c_{1}b(x)|^{q}\right)^{\frac{1}{q}}.
\]
Note that $\|f\|_{BMO_{b}^{q}}=0$ if and only if $f=\alpha+\beta b$
and hence in order to consider $\|f\|_{BMO_{b}^{q}}$ as a norm one
needs to take quotient by the subspace $\langle1,b\rangle$ (the space of linear combinations of $1$ and $b$). It readily
follows from the definition that $BMO\subset BMO_{b}^{q}$ for every
$q$. It is also easy to show that if $b\in BMO$ then $b^{2}\in BMO_{b}^{q}.$
And for instance choosing $b(x)=\log(x)$, we have that $\log(x)^{2}\in BMO_{b}^{q}\setminus BMO$. 

Inspired by the definition of $BMO_{b}^{q}$ we provide the following
result for commutators.
\begin{thm}
\label{thm:OscLinear}Let $b\in BMO$ and $T$ a Calderón-Zygmund
operator satisfying a$\log$-Dini regularity condition. Then
for every ball $B$ be a ball, if $\delta\in(0,1)$ and $r>1$ we
have that
\[
\inf_{c_{1}\in\mathbb{R}}\left( \fint_{B}|[b,T]f(x)-c_{1}-T(f\chi_{(2B)^{c}})(c_{B})b(x)|^{\delta}dx\right)^{\frac{1}{\delta}}\lesssim r'\left\Vert \frac{f}{w}\right\Vert _{L^{\infty}}\|b\|_{BMO}\inf_{z\in B}M_{r}w(z).
\]
Consequently
\[
\inf_{c_{1},c_{2}\in\mathbb{R}}\left( \fint_{B}|[b,T]f(x)-c_{1}-c_{2}b(x)|^{\delta}dx\right)^{\frac{1}{\delta}}\lesssim r'\left\Vert \frac{f}{w}\right\Vert _{L^{\infty}}\|b\|_{BMO}\inf_{z\in B}M_{r}w(z).
\]
 
\end{thm}

The next natural question would be whether it is possible to extrapolate
from the condition above. We show that that is the case under some
additional conditions.
\begin{thm}
\label{thm:ExtrapLinear}Let $T$ be a linear operator such that for
every $b\in BMO$ and every $w\in A_{1}$ 
\[
\inf_{c_{1}\in\mathbb{R}}\left( \fint_{B}|[b,T]f(x)-c_{1}-T(f\chi_{(2B)^{c}})(c_{B})b(x)|^{\delta}dx\right)^{\frac{1}{\delta}}\leq c_{w,T}\left\Vert \frac{f}{w}\right\Vert _{L^{\infty}}\|b\|_{BMO}\inf_{z\in B}w(z)
\]
and such that Lerner's grand maximal operator
\[
\mathcal{M}_{T}f(x)=\sup_{x\in B}\esssup_{z\in B}T(f\chi_{(2B)^{c}})(z)
\]
is bounded on $L^{p}(v)$ for some $p\in(1,\infty)$ and some $v\in A_{p}$. Then 
\[
\left\Vert [b,T]f\right\Vert _{L^{p}(v)}\leq c_{v,T}\|b\|_{BMO}\|f\|_{L^{p}(v)}.
\]
\end{thm}

Observe that the operator $\mathcal{M}_{T}$ was introduced in \cite{Le}
in order to study sparse domination. There it was shown that in the
case of $T$ being a Calderón-Zygmund operator
\[
\mathcal{M}_{T}f(x)\lesssim Mf(x)+T^{*}f(x),
\]
where $T^{*}$ stands for the maximal Calderón-Zygmund operator. Since
both $M$ and $T^{*}$ are bounded on $L^{p}(w)$ for $w\in A_{p}$,
the result above combined with the estimate in Theorem \ref{eq:OscT}
allows to provide an alternative proof of the weighted $L^{p}$ boundedness
of the commutator $[b,T]$. 

Here we just presented the results in the linear setting. However
results in the multilinear setting are feasible as well and will be
obtained in Section \ref{sec:MlC}.

The remainder of the paper is organized as follows. In Section \ref{sec:Preliminaries}
we gather some preliminaries.

In Section \ref{sec:PM} we settle Theorems \ref{thm:OscLinear} and
\ref{thm:ExtrapLinear}. Finally in Section \ref{sec:MlC} we present
and settle the multilinear counterparts of the main results.

\section{Preliminaries\label{sec:Preliminaries}}

We recall that $T$ is a Calderón-Zygmund operator if $T$ is a linear
operator that is bounded on $L^{2}$ and it admits a representation
in terms of a kernel $K$ 
\[
Tf(x)=\int_{\mathbb{R}^{n}}K(x,y)f(y)dy\qquad x\not\in\supp f
\]
where $K$ satisfies the following properties.
\begin{enumerate}
\item[] Size~condition: $|K(x,y)|\leq C_{K} |x-y|^{-n}$;\\
\item[] Smoothness~condition: Provided that $|x-y|\geq2|x-z|$, 
\[
|K(x,y)-K(z,y)|+|K(y,x)-K(y,z)|\leq\omega\left(\frac{|x-z|}{|x-y|}\right)\frac{1}{|x-y|^n},
\]
where $\omega$ is a continuous subadditive function such that $$\int_{0}^{1}\omega(t)\log\left(\frac{1}{t}\right)\frac{dt}{t}<\infty.$$ 
\end{enumerate}

In the definition of commutators we used $BMO$ functions. We recall
that $b\in BMO$ if 
\[
\|b\|_{BMO}=\sup_{B} \fint_{B}|b-b_{B}|<\infty.
\]
A fundamental property of this space of functions is the well known
John-Nirenberg that says that the integrability of the oscillations
self-improves to exponential integrability, namely, there exist constants
$\lambda,c>0$ such that for every ball $B$ and every $BMO$ function
\[
|\{x\in B: |b-b_B|> \lambda\}|\lesssim e^{-c \lambda/{\|b\|_{BMO}}} |B|.
\]
Note that this in turn implies that 
\begin{equation}\label{eq:jn}
\fint_B |b-b_B|^\alpha \lesssim \max\{\alpha,1\} \|b\|_{BMO}
\end{equation}
for every $\alpha>0$. Another fact that we will use in the sequel
is that if $B$ is a ball then 
\begin{equation}
|b_{2^{j}B}-b_{B}|\lesssim j\|b\|_{BMO}.\label{eq:prop2jBBbmo}
\end{equation}
We remit the interested reader to \cite{J} for more details on $BMO$.

Quite related to the definition of $BMO$ is that of the sharp maximal
function. Given $\delta>0$, we define 
\[
M_{\sharp,\delta}(f)(x)=\sup_{x\in B}\inf_{c\in\mathbb{R}}\left( \fint_{B}|f-c|^{\delta}\right)^{\frac{1}{\delta}}.
\]

We would like to end this preliminaries section by gathering some basic
facts about multilinear theory. We recall that a linear operator $T$
is an $m$-linear Calderón-Zygmund operator if $T:L^{p_{1}}\times\dots\times L^{p_{m}}\rightarrow L^{p}$
for some $1<p_{1},\dots,p_{m}<\infty$ with $\frac{1}{p}=\sum_{i=1}^{m}\frac{1}{p_{i}}$
and it admits the following representation 
\[
T(\vec{f})(x)=\int_{\mathbb{R}^{nm}}K(x,y_{1},\dots,y_{m})f(y_{1})\dots f(y_{m})dy_{1}\dots dy_{m}
\]
where $x\not\in\supp(f_{i})$ for any $i\in\{1,\dots,m\}$, in terms
of a kernel $K$ that satisfies the following properties.
\begin{enumerate}
\item [] Size~condition: $|K(x,\vec{y})|\leq C_{K} (\sum_{i=1}^{m}|x-y_{i}|)^{-mn}$;
\item [] Smoothness~condition: Given $\omega$ a continuous subadditive
function such that $\int_{0}^{1}\omega(t)\log\left(\frac{1}{t}\right)\frac{dt}{t}<\infty$,
the following conditions hold
\[
|K(x,\vec{y})-K(z,\vec{y})|\leq\omega\left(\frac{|x-z|}{\max_{i\in\{1,\dots,m\}}|x-y_{i}|}\right)\frac{1}{\left(\sum_{i=1}^{m}|x-y_{i}|\right)^{mn}}
\]
provided that $\max_{i\in\{1,\dots,m\}}|x-y_{i}|\geq2|x-z|$, and
also, for any $j\in\{1,\dots,m\}$
\begin{align*}
  |K(x,y_{1},\dots,y_{j},\dots, y_{m})&-K(x,y_{1},\dots,y_{j}',\dots, y_{m})|\\
 & \leq\omega\left(\frac{|y_{j}-y_{j}'|}{\max_{i\in\{1,\dots,m\}}|x-y_{i}|}\right)\frac{1}{\left(\sum_{i=1}^{m}|x-y_{i}|\right)^{mn}}
\end{align*}
where $\max_{i\in\{1,\dots,m\}}|x-y_{i}|\geq2|y_{j}-y'_{j}|$.
\end{enumerate}
Note that in this context the commutator $[b,T]_{j}\vec{f}(x)$ is
defined as 
\[
[b,T]_{j}\vec{f}(x)=b(x)T(\vec{f})(x)-T(f_{1},\dots,f_{j}b,\dots,f_{m}).
\]
Note that the definition is essentially equivalent whichever index
we commute in. Hence throughout the remainder of this work we will
consider just the case $[b,T]_{1}$. 

Let us also recall that we say $\vec w=(w_1,\dots, w_m)\in A_{\vec p}$, if 
\[
\sup_Q \Big(\fint_Q w^p\Big)^{\frac 1p}\prod_{i=1}^m\Big(\fint_Q w_i^{-p_i'}\Big)^{\frac 1{p_i'}}<\infty,\qquad w:=\prod_{i=1}^m w_i
\]
 where $\vec p=(p_1, \dots, p_m)$ with $1\le p_i\le \infty$ and $1/p=1/{p_1}+\dots+1/{p_m}$. A consequence of the multilinear extrapolation result that appeared first in \cite[Theorem 4.12]{N} (see as well  \cite{LMMOV}) states that 
 \begin{thm}
 Let $(f, f_1,\dots, f_m)$ be an $(m+1)$-tuple of functions. Suppose
that 
\[
\|fw\|_{L^{\infty}}\le c_{\vec w}\prod_{i=1}^m\|f_iw_i\|_{L^{\infty}}
\]
holds for all $\vec w$ with $\vec w\in A_{(\infty,\dots, \infty)}$, where $c_{\vec w}$ depends
only on $[\vec w]_{A_{(\infty,\dots, \infty)}}$. Then for all $\vec p$ with $p_i>1$, $i=1,\dots, m$, and all $\vec w\in A_{\vec p}$,
we have 
\[
\|fw\|_{L^{p}}\le c_{\vec w}\prod_{i=1}^m\|f_iw_i\|_{L^{p_i}},
\]
where $\tilde{c}_{\vec w}$ depends only on $[\vec w]_{A_{\vec p}}$.
 \end{thm}

\section{Proofs of the main results\label{sec:PM}}

\subsection{Proof of Theorem \ref{thm:OscLinear}}

Let $B$ be a ball and $c_{2},\lambda$ constants to be chosen. Let
\[
c_{1}=-\lambda Tf_{2}(c_B)-T((b-\lambda)f_{2})(c_{B}),
\]where $f_2=f\chi_{\mathbb R^n\setminus 2B}$. 
Then we begin arguing as follows.
\begin{align*}
  &\Big( \fint_{B}|[b,T]f(x)-c_{1}-c_{2}b(x)|^{\delta}dx\Big)^{\frac{1}{\delta}}\\
 & =\left( \fint_{B}|[b-\lambda,T]f(x)-c_{1}-c_{2}b(x)|^{\delta}dx\right)^{\frac{1}{\delta}}\\
 & \lesssim\left( \fint_{B}|(b(x)-\lambda)Tf(x)+\lambda Tf_{2}(c_B)-c_{2}b(x)|^{\delta}dx\right)^{\frac{1}{\delta}}\\
 & \hspace{4cm}+\left( \fint_{B}|T((b-\lambda)f)(x)-T((b-\lambda)f_{2})(c_{B})|^{\delta}dx\right)^{\frac{1}{\delta}}\\
 & =:L_{1}+L_{2}
\end{align*}
Note that for $L_{1}$, choosing $\lambda=b_{2B}$ we have that for
$\delta<\varepsilon<1$, calling $f_{1}=f\chi_{2B}$,
\begin{align*}
L_{1} & =\left(\fint_{B}|(b(x)-b_{2B})Tf(x)+b_{2B} Tf_{2}(c_B)-c_{2}b(x)|^{\delta}dx\right)^{\frac{1}{\delta}}\\
 & \lesssim\left(\fint_{B}|(b(x)-b_{2B})Tf_{1}(x)|^{\delta}dx\right)^{\frac{1}{\delta}}\\
 &\hspace{3cm}+\left(\fint_{B}|(b(x)-b_{2B})Tf_{2}(x)+b_{2B} Tf_{2}(c_B)-c_{2}b(x)|^{\delta}dx\right)^{\frac{1}{\delta}}\\
 & =:L_{11}+L_{12}.
\end{align*}
First we focus on $L_{11}$. We argue as follows
\begin{align*}
L_{11} & =\left( \fint_{B}|(b(x)-b_{2B})Tf_{1}(x)|^{\delta}dx\right)^{\frac{1}{\delta}}\\
 & \leq\left( \fint_{B}|b(x)-b_{2B}|^{\delta(\frac{\varepsilon}{\delta})'}dx\right)^{\frac{1}{\delta(\frac{\varepsilon}{\delta})'}}\left( \fint_{B}|Tf_{1}|^{\varepsilon}dx\right)^{\frac{1}{\varepsilon}}\\
 & \lesssim\|b\|_{BMO} \fint_{2B}|f|=\|b\|_{BMO} \fint_{2B}\frac{|f|}{w}w\\
 & \leq\|b\|_{BMO}\left\Vert \frac{f}{w}\right\Vert _{L^{\infty}}\inf_{z\in B}Mw.
\end{align*}
Now we turn to $L_{12}$. Choosing $c_2=Tf_{2}(c_B)$ we have that
\begin{align*}
L_{12} & =\left( \fint_{B}|(b(x)-b_{2B})Tf_{2}(x)-Tf_{2}(c_B)b(x)+b_{2B} Tf_{2}(c_B)|^{\delta}dx\right)^{\frac{1}{\delta}}\\
 & =\left( \fint_{B}|(b(x)-b_{2B})Tf_{2}(x)-(b(x)-b_{2B})Tf_{2}(c_B)|^{\delta}dx\right)^{\frac{1}{\delta}}\\
 & \leq\left( \fint_{B}|b(x)-b_{2B}||Tf_{2}(x)-Tf_{2}(c_B)|dx\right).
\end{align*}
From this point taking into account the smoothness condition of the
kernel we may argue as follows
\begin{align*}
 & \left( \fint_{B}|b(x)-b_{2B}||Tf_{2}(x)-Tf_{2}(c_B)|dx\right)\\
 & \leq \fint_{B}|b(x)-b_{2B}|\int_{\mathbb R^n\setminus 2B}|K(x,y)-K(c_B,y)||f(y)|dydx\\
 & \leq \fint_{B}|b(x)-b_{2B}|\int_{\mathbb R^n\setminus 2B }\frac{1}{|x-y|^{n}}\omega\left(\frac{|x-c_B|}{|x-y|}\right)|f(y)|dydx\\
 & \leq \fint_{B}|b(x)-b_{2B}|\sum_{j=1}^{\infty}\int_{2^{j+1}B\setminus2^{j}B}\frac{1}{|x-y|^{n}}\omega\left(\frac{|x-c_B|}{|x-y|}\right)|f(y)|dydx\\
 & \le  \fint_{B}|b(x)-b_{2B}|dx\sum_{j=1}^{\infty}\frac{1}{2^{jn}l(B)^{n}}\omega\left(\frac{l(B)}{2^{j}l(B)}\right)\int_{2^{j+1}B\setminus2^{j}B}|f(y)|dy\\
 & \lesssim \|b\|_{BMO}\sum_{j=1}^{\infty}\omega(2^{-j}) \fint_{2^{j+1}B}|f(y)|dy\\
 & \lesssim\|b\|_{BMO}\left\Vert \frac{f}{w}\right\Vert _{L^{\infty}}\inf_{z\in B}Mw.
\end{align*}
We continue bounding $L_{2}$. Note that
\begin{align*}
L_{2} & \lesssim\left(\fint_{B}|T((b-b_{2B})f_{1})(x)|^{\delta}dx\right)^{\frac{1}{\delta}}\\
 & +\left(\fint_{B}|T((b-b_{2B})f_{2})(x)-T((b-b_{2B})f_{2})(c_{B})|^{\delta}dx\right)^{\frac{1}{\delta}}\\
 & =L_{21}+L_{22}.
\end{align*}
For $L_{21}$ by Kolmogorov inequality,
\begin{align*}
L_{21}  &=\left(\fint_{B}|T((b-b_{2B})f_{1})(x)|^{\delta}dx\right)^{\frac{1}{\delta}}
  \lesssim \fint_{2B}|b-b_{2B}||f|dx\\
 & \leq\left\Vert \frac{f}{w}\right\Vert _{L^{\infty}}\fint_{2B}|b-b_{2B}|wdx \lesssim r'\left\Vert \frac{f}{w}\right\Vert _{L^{\infty}}\|b\|_{BMO}\inf_{z\in B}M_{r}w(z),
\end{align*}where in the last step we have used H\"older's inequality and \eqref{eq:jn}. 
For $L_{22}$, we have that, using the smoothness condition of the
kernel,
\begin{align*}
L_{22} & \leq\left(\fint_{B}|T((b-b_{2B})f_{2})(x)-T((b-b_{2B})f_{2})(c_{B})|dx\right)\\
 & \leq \fint_{B}\int_{\mathbb R^n \setminus 2B}|K(x,y)-K(c_{B},y)||b(y)-b_{2B}||f(y)|dydx\\
 & \leq \fint_{B}\int_{\mathbb R^n\setminus 2B }\omega\left(\frac{|x-c_{B}|}{|x-y|}\right)\frac{1}{|x-y|^{n}}|b(y)-b_{2B}||f(y)|dydx\\
 & \leq\sum_{j=1}^{\infty}\omega(2^{-j})\frac{1}{(2^{j}l(B))^{n}}\int_{2^{j+1}B\setminus2^{j}B}|b(y)-b_{2B}||f(y)|dy\\
 & \lesssim \left\Vert \frac{f}{w}\right\Vert _{L^{\infty}}\sum_{j=1}^{\infty}\omega(2^{-j})\fint_{2^{j+1}B}|b(y)-b_{2B}|w(y)dy\\
 & \leq\left\Vert \frac{f}{w}\right\Vert _{L^{\infty}}\sum_{j=1}^{\infty}\omega(2^{-j}) \fint_{2^{j+1}B}|b(y)-b_{2^{j+1}B}|w(y)dy\\
 & \hspace{3cm}+\left\Vert \frac{f}{w}\right\Vert _{L^{\infty}}\sum_{j=1}^{\infty}\omega(2^{-j})|b_{2^{j+1}B}-b_{B}| \fint_{2^{j+1}B}w(y)dy\\
 & \lesssim r' \left\Vert \frac{f}{w}\right\Vert _{L^{\infty}}\|b\|_{BMO}\essinf_{z\in B}M_{r}w(z),
\end{align*}where in the last step we have used H\"older's inequality, \eqref{eq:jn} and \eqref{eq:prop2jBBbmo}.
This ends the proof.

\subsection{Proof of Theorem \ref{thm:ExtrapLinear}.}

Let us fix a ball $B$ and $x\in B$. Following the same notation as that in the proof of Theorem \ref{thm:OscLinear}, if we choose
$c=c_{1}+Tf_{2}(c_{B})b_{2B}$, 
\begin{align*}
 & \Big(\fint_{B}|[b,T]f(y)-c|^{\delta}dy\Big)^{\frac{1}{\delta}}\\
 & =\Big(\fint_{B}|[b,T]f(y)-c_{1}-Tf_{2}(c_{B})b(y)-Tf_{2}(c_{B})b_{2B}+Tf_{2}(c_{B})b(y)|^{\delta}dy\Big)^{\frac{1}{\delta}}\\
 & \lesssim\Big(\fint_{B}|[b,T]f(y)-c_{1}-Tf_{2}(c_{B})b(y)|^{\delta}dy\Big)^{\frac{1}{\delta}}+|Tf_{2}(c_{B})|\Big(\fint_{B}|b(y)-b_{2B}|^{\delta}dy\Big)^{\frac{1}{\delta}}\\
 & \lesssim\Big(\fint_{B}|[b,T]f(y)-c_{1}-Tf_{2}(c_{B})b(y)|^{\delta}dy\Big)^{\frac{1}{\delta}}+\mathcal{M}_{T}f(x)\|b\|_{BMO}.
\end{align*}
Note that this yields that 
\[
M_{\sharp,\delta}([b,T]f)(x)\lesssim\sup_{x\in B}\inf_{c_{1}\in\mathbb{R}}\Big(\fint_{B}|[b,T]f(y)-c_{1}-Tf_{2}(c_{B})b(y)|^{\delta}dy\Big)^{\frac{1}{\delta}}+\mathcal{M}_{T}f(x)\|b\|_{BMO}.
\]
Observe that if we call 
\[
g(x):=\sup_{x\in B}\inf_{c_{1}\in\mathbb{R}}\Big(\fint_{B}|[b,T]f(y)-c_{1}-Tf_{2}(c_{B})b(y)|^{\delta}dy\Big)^{\frac{1}{\delta}}
\]
by hypothesis we have that 
\[
\|gw\|_{L^{\infty}}\le c_{w}\|b\|_{BMO}\|fw\|_{L^{\infty}}
\]
and hence by Theorem \ref{thm:ExtrapInfty} we have that for all
$1<q<\infty$ and every $w\in A_{q}$ 
\[
\|g\|_{L^{q}(w)}\le\tilde{c}_{w}\|b\|_{BMO}\|f\|_{L^{q}(w)}.
\]
Since by hypothesis as well we know that 
\[
\|\mathcal{M}_{T}f\|_{L^{p}(v)}\le\tilde{c}_{v}\|f\|_{L^{p}(v)}.
\]
We have that combining the estimates above,
\[
\|M_{\sharp,\delta}([b,T]f)\|_{L^{p}(v)}\le\tilde{c}_{v}\|b\|_{BMO}\|f\|_{L^{p}(v)}.
\]
Then the desired estimate 
\[
\|[b,T]f\|_{L^{p}(v)}\le\tilde{c}_{v}\|b\|_{BMO}\|f\|_{L^{p}(v)}.
\]
follows from the Fefferman-Stein's inequality.

\section{Multilinear counterparts\label{sec:MlC}}

In this section we present multilinear versions of the results presented
above. We begin providing a counterpart of Theorem \ref{thm:OscLinear}.

\begin{thm}
\label{lem:multi}Let $b\in BMO$ and $T$ be an $m$-linear $CZO$.
Then for every ball $B$, if $\delta\in(0,1/m)$ then 
\begin{align*}
 & \inf_{c_{1}\in\mathbb{R}}\Big(\fint_{B}|[b,T]_{1}(\vec{f})(y)-c_{1}-\left(T(\vec{f})-T(\vec{f}\chi_{2B})\right)b(y)|^{\delta}dy\Big)^{\frac{1}{\delta}}\\
 & \lesssim c_{w}\|b\|_{BMO}\prod_{i=1}^{\infty}\|f_{i}w_{i}\|_{L^{\infty}}\essinf_{x\in B}\frac{1}{w(x)},
\end{align*}
where $(w_{1},\ldots,w_{m})\in A_{(\infty,\ldots,\infty)}$ and $w=\prod_{i=1}^{m}w_{i}$.
Consequently, the following inequality holds as well
\[
\inf_{c_{1},c_{2}\in\mathbb{R}}\Big(\fint_{B}|[b,T]_{1}(\vec{f})(y)-c_{1}-c_{2}b(y)|^{\delta}dy\Big)^{\frac{1}{\delta}}\lesssim c_{w}\|b\|_{BMO}\prod_{i=1}^{\infty}\|f_{i}w_{i}\|_{L^{\infty}}\essinf_{x\in B}\frac{1}{w(x)},
\]
\end{thm}
\begin{proof}[Proof of Theorem \ref{lem:multi}]
For notational convenience we may denote 
\[
T_{B}(\vec{f})=T(\vec{f})-T(\vec{f}\chi_{2B}).
\]
Let 
\begin{align*}
c_{1}(B) & =-b_{2B}T_{B}(\vec{f})(c_{B})-T_{B}\big((b-b_{2B})f_{1},f_{2},\ldots,f_{m}\big)(c_{B})
\end{align*}
and 
\begin{align*}
c_{2}(B)=T_{B}(\vec{f})(c_{B}).
\end{align*}
Having that notation in mind we have that 
\begin{align*}
\Big(\fint_{B} & |[b,T]_{1}(\vec{f})(y)-c_{1}(B)-c_{2}(B)b(y)|^{\delta}dy\Big)^{\frac{1}{\delta}}\\
 & =\Big(\fint_{B}|[b-b_{2B},T]_{1}(\vec{f})(y)-c_{1}(B)-c_{2}(B)b(y)|^{\delta}dy\Big)^{\frac{1}{\delta}}\\
 & \lesssim\Big(\fint_{B}|(b-b_{2B})T(\vec{f})(y)-(b-b_{2B})T_{B}(\vec{f})(c_{B})|^{\delta}dy\Big)^{\frac{1}{\delta}}\\
 & \hspace{1cm}+\Big(\fint_{B}|T((b-b_{2B})f_{1},\ldots,f_{m})(y)-T_{B}((b-b_{2B})f_{1},\ldots,f_{m})(c_{B})|^{\delta}dy\Big)^{\frac{1}{\delta}}\\
 & =:I_{1}+I_{2}.
\end{align*}
For $I_{1}$, we have 
\begin{align*}
I_{1} & \lesssim\Big(\fint_{B}|b-b_{2B}|^{\delta}\cdot|T_{B}(\vec{f})(y)-T_{B}(\vec{f})(c_{B})|^{\delta}dy\Big)^{\frac{1}{\delta}}+\Big(\fint_{B}|b-b_{2B}|^{\delta}\cdot|T(\vec{f}\chi_{2B})(y)|^{\delta}dy\Big)^{\frac{1}{\delta}}\\
 & \lesssim\|b\|_{BMO}\Big(\sup_{y\in B}|T_{B}(\vec{f})(y)-T_{B}(\vec{f})(c_{B})|+\|T(\vec{f}\chi_{2B})\|_{L^{\frac{1}{m},\infty}(B,\frac{dx}{|B|})}\Big)\\
 & \lesssim\|b\|_{BMO}\Big(\sum_{k=1}^{\infty}\omega(2^{-k})\prod_{i=1}^{m}\fint_{2^{k}B}|f_{i}|+\prod_{i=1}^{m}\fint_{2B}|f_{i}|\Big)\\
 & \le\|b\|_{BMO}[\vec{w}]_{A_{(\infty,\ldots,\infty)}}\Big(\prod_{i=1}^{\infty}\|f_{i}w_{i}\|_{L^{\infty}}\Big)\essinf_{x\in B}\frac{1}{w(x)},
\end{align*}
where we have used the weak endpoint estimate of $T$. Now we turn
to estimate $I_{2}$. We argue as follows.
\begin{align*}
I_{2} & \lesssim\Big(\fint_{B}|T_{B}((b-b_{2B})f_{1},\ldots,f_{m})(y)-T_{B}((b-b_{2B})f_{1},\ldots,f_{m})(c_{B})|^{\delta}dy\Big)^{\frac{1}{\delta}}\\
 & \hspace{1cm}+\Big(\fint_{B}|T((b-b_{2B})f_{1}\chi_{2B},\ldots,f_{m}\chi_{2B})(y)|^{\delta}dy\Big)^{\frac{1}{\delta}}\\
 & =:I_{21}+I_{22}.
\end{align*}
The estimate of $I_{22}$ can be handled similarly as before, that
is, we use Kolmogorov inequality and then the weak type endpoint estimate.
\begin{align*}
I_{22} & \lesssim\Big(\fint_{2B}|b-b_{2B}||f_{1}|\Big)\prod_{i=2}^{m}\fint_{2B}|f_{i}|\\
 & \le[\vec{w}]_{A_{(\infty,\ldots,\infty)}}\Big(\prod_{i=1}^{m}\|f_{i}w_{i}\|_{L^{\infty}}\Big)\essinf_{x\in B}\frac{1}{w(x)}\cdot\frac{1}{w_{1}^{-1}(2B)}\int_{2B}|b-b_{2B}|w_{1}^{-1}\\
 & \lesssim\|b\|_{BMO}[\vec{w}]_{A_{(\infty,\ldots,\infty)}}\Big(\prod_{i=1}^{m}\|f_{i}w_{i}\|_{L^{\infty}}\Big)\essinf_{x\in B}\frac{1}{w(x)},
\end{align*}
where the last inequality holds since $w_{1}^{-1}\in A_{\infty}$.

It remains to consider $I_{21}$. Similarly as before we have that 
\begin{align*}
I_{21} & \lesssim\sum_{k=1}^{\infty}\omega(2^{-k})\fint_{2^{k}B}|b-b_{2B}||f_{1}|\prod_{i=2}^{m}\fint_{2^{k}B}|f_{i}|\\
 & \le\sum_{k=1}^{\infty}\omega(2^{-k})\fint_{2^{k}B}|b-b_{2^{k}B}||f_{1}|\prod_{i=2}^{m}\fint_{2^{k}B}|f_{i}|+\sum_{k=1}^{\infty}\omega(2^{-k})|b_{2B}-b_{2^{k}B}|\prod_{i=1}^{m}\fint_{2^{k}B}|f_{i}|\\
 & \lesssim\|b\|_{BMO}[\vec{w}]_{A_{(\infty,\ldots,\infty)}}\Big(\prod_{i=1}^{m}\|f_{i}w_{i}\|_{L^{\infty}}\Big)\essinf_{x\in B}\frac{1}{w(x)},
\end{align*}
where we have used that 
\[
|b_{2B}-b_{2^{k}B}|\lesssim k\|b\|_{BMO}.
\]
This requires that the kernel satisfies the log-Dini condition. This
completes the proof.
\end{proof}
Having the Theorem above at our disposal we can obtain the following
result.
\begin{thm}
\label{thm:ExtrapMulti}Let $T$ be an $m$-linear operator such that for
every $b\in BMO$ and every $\vec w\in A_{(\infty,\dots, \infty)}$, 
\begin{align*}
 & \inf_{c_{1}\in\mathbb{R}}\Big(\fint_{B}|[b,T]_{1}(\vec{f})(y)-c_{1}-\left(T(\vec{f})-T(\vec{f}\chi_{2B})\right)(y)b(y)|^{\delta}dy\Big)^{\frac{1}{\delta}}\\
 & \lesssim c_{w}\|b\|_{BMO}\prod_{i=1}^{\infty}\|f_{i}w_{i}\|_{L^{\infty}}\essinf_{x\in B}\frac{1}{w(x)}
\end{align*}
and such that Lerner's grand maximal operator
\[
\mathcal{M}_{T}(\vec f)(x)=\sup_{x\in B}\esssup_{z\in B}T(\vec{f}\chi_{(2B)^{c}})(z)
\]
satisfies
\[
\|\mathcal{M}_{T}(\vec f)v\|_{L^{p}}\le c_{\vec v}\prod_{i=1}^m\|f_iv_i\|_{L^{p_i}}
\]for 
some $\vec p$ with $p_i>1$, $i=1,\dots, m$, 
and some $\vec v\in A_{\vec p}$. Then 
\[
\| [b,T]_1(\vec f) v\|_{L^{p}}\leq c_{\vec{v},T}\|b\|_{BMO}\prod_{i=1}^{m}\|f_i v_i\|_{L^{p_i}}.
\]
 \end{thm}

\begin{proof}
The argument is analogous to the one given for the linear case. Let
us define 
\[
g(x)=\sup_{B\ni x}\inf_{c_{1}\in\mathbb{R}}\Big(\fint_{B}|[b,T]_{1}(\vec{f})(y)-c_{1}-\left(T(\vec{f})-T(\vec{f}\chi_{2B})\right)(y)b(y)|^{\delta}dy\Big)^{\frac{1}{\delta}},
\]
then, arguing as in the linear case, we have that
\begin{align*}
M_{\sharp,\delta}([b,T]_{1}(\vec{f}))(x) & \le\sup_{B\ni x}\Big(\fint_{B}|[b,T]_{1}(\vec{f})(y)-c(B)|^{\delta}dy\Big)^{\frac{1}{\delta}}\\
 & \lesssim g(x)+\sup_{B\ni x}\Big(\fint_{B}|b(y)-b_{2B}|^{\delta}dy\Big)^{\frac{1}{\delta}}|T_{B}(\vec{f})(c_{B})|\\
 & \lesssim g(x)+\|b\|_{BMO}\mathcal{M}_{T}(\vec{f})(x).
\end{align*}
To deal with $g$ we argue as in the linear setting. $\mathcal{M}_{T}$
is bounded by hypothesis. Hence we are done.
\end{proof}
Note that in the case $T$ being a Calderón-Zygmund operator, in the
bilinear case, careful calculus to bound $\mathcal{M}_T$ was presented in \cite{Li}. Such an estimate, that we recall in the following line, 
can be extended to the multilinear case directly. Therefore, we have
\[
\mathcal{M}_{T}(\vec{f})(x)\lesssim M(\vec{f})(x)+M_{s}(T(\vec{f}))(x)
\]
for every $0<s<\frac{1}{m}$. Of course, since $M_{s}$ is increasing
with $s$, the inequality holds for all $s>0$. In particular, we
can choose $s=\frac{1}{m}$, which in turn allows us to show that 
\[
\|M_{1/m}(T(\vec{f}))w\|_{L^{p}}=\|M(T(\vec{f})^{1/m})\|_{L^{mp}(w^{p})}^{m}\lesssim\|T(\vec{f})w\|_{L^{p}}\lesssim\prod_{i=1}^{m}\|f_{i}w_{i}\|_{L^{p}},
\]
where we have used the fact that if $(w_{1},\ldots,w_{m})\in A_{\vec{p}}$
then $w^{p}\in A_{mp}$ (when $p<\infty$). This ends the argument and completes an alternative proof of the boundedness of $[b,T]_i$ in the multilinear setting.

\bibliographystyle{abbrvnat}

{\footnotesize\bibliography{Refs}}

\end{document}